\documentclass[12pt, article]{amsart}
\usepackage{amsmath, amsthm, amscd, amsfonts, amssymb, graphicx, color}
\usepackage[bookmarksnumbered, colorlinks, plainpages]{hyperref}
\hypersetup{colorlinks=true,linkcolor=red, anchorcolor=green, citecolor=cyan, urlcolor=red, filecolor=magenta, pdftoolbar=true}

\newtheorem{theorem}{Theorem}[section]
\newtheorem{lemma}[theorem]{Lemma}
\newtheorem{proposition}[theorem]{Proposition}
\newtheorem{corollary}[theorem]{Corollary}

\theoremstyle{definition}

\theoremstyle{remark}

\newcommand{\tr}{{\rm{tr}}}
\usepackage[all]{xy}
\numberwithin{equation}{section}

\newtheorem*{theorem*}{Theorem}

\begin{document}
\title[Advanced refinements of Young and Heinz inequalities]
{Advanced refinements of Young and Heinz inequalities}
\author[M. Sababheh]{M. Sababheh}
\address{Department of Basic Sciences, Princess Sumaya University For Technology, Al Jubaiha, Amman 11941, Jordan.\newline
 Department of Mathematics, University of Sharjah, Sharjah, United Arab Emirates.}
\email{sababheh@psut.edu.jo, msababheh@sharjah.ac.ae, sababheh@yahoo.com}

\author[M.S. Moslehian]{M. S. Moslehian}
\address{Department of Pure Mathematics, Ferdowsi University of Mashhad, P. O. Box 1159, Mashhad 91775, Iran}
\email{moslehian@um.ac.ir}


\maketitle
\begin{abstract}
In this article, we prove several multi-term refinements of Young type inequalities for both real numbers and operators improving several known results. Among other results, we prove
\begin{eqnarray*}
A\#_{\nu}B&+&\sum_{j=1}^{N}s_{j}(\nu)\left(A\#_{\alpha_j(\nu)}B+A\#_{2^{1-j}+\alpha_j(\nu)}B-2A\#_{2^{-j}+\alpha_j(\nu)}B\right)\leq A\nabla_{\nu}B,
\end{eqnarray*}
for the positive operators $A$ and $B$, where $0\leq \nu\leq 1, N\in\mathbb{N}$ and $\alpha_j(\nu)$ is a certain function. Moreover, some new Heinz type inequalities involving the Hilbert-Schmidt norm are established.
\end{abstract}

\section{introduction}

The simple inequality
\begin{equation}\label{original_young}
a^{\nu}b^{1-\nu}\leq \nu a+(1-\nu)b, a,b>0, 0\leq \nu\leq 1
\end{equation}
is the celebrated Young inequality. Even though this inequality looks very simple, it is of great interest in operator theory. We refer the reader to \cite{kittanehmanasreh,kittanehmanasreh2,omarkittaneh,sabchoiref,sabLM, ALE,zuo} as a sample of the extensive use of this inequality in this field. Refining this inequality has taken the attention of many researchers in the field, where adding a positive term to the left side is possible.

Among the first refinements of this inequality is the squared version proved in \cite{omarkittaneh}
\begin{equation}\label{hirzkitt}
(a^{\nu}b^{1-\nu})^2+\min\{\nu,1-\nu\}^2(a-b)^2\leq (\nu a+(1-\nu)b)^2.
\end{equation}
Later, the authors in \cite{kittanehmanasreh} obtained the other interesting refinement
\begin{equation}\label{kittmans}
a^{\nu}b^{1-\nu}+\min\{\nu,1-\nu\}(\sqrt{a}-\sqrt{b})^2\leq\nu a+(1-\nu)b.
\end{equation}
A common fact about the above refinements is having one refining term.\\
In the recent paper \cite{zhao}, some reverses and refinements of Young's inequality were presented, with one or two refining terms. In particular, it is proved that
\begin{equation}
\left\{\begin{array}{cc}a^{\nu}b^{1-\nu}+\nu(\sqrt{a}-\sqrt{b})^2+r_0(\sqrt[4]{ab}-\sqrt{b})^2\leq \nu a+(1-\nu)b,&0<\nu\leq\frac{1}{2}\\
a^{\nu}b^{1-\nu}+(1-\nu)(\sqrt{a}-\sqrt{b})^2+r_0(\sqrt[4]{ab}-\sqrt{a})^2\leq \nu a+(1-\nu)b,&\frac{1}{2}<\nu\leq 1\end{array}\right.,
\end{equation}
where $r_0=\min\{2r,1-2r\}$ for $r=\min\{\nu,1-\nu\}.$ These inequalities refine (\ref{kittmans}) by adding a second refining term to the original Young's inequality. In the same paper, the following reversed versions have been proved too.
\begin{equation}\label{reverse_zhao}
\left\{\begin{array}{cc}\nu a+(1-\nu)b+r_0(\sqrt[4]{ab}-\sqrt{a})^2\leq a^{\nu}b^{1-\nu}+(1-\nu)(\sqrt{a}-\sqrt{b})^2,&0\leq\nu\leq\frac{1}{2}\\
\nu a+(1-\nu)b+r_0(\sqrt[4]{ab}-\sqrt{b})^2\leq a^{\nu}b^{1-\nu}+\nu(\sqrt{a}-\sqrt{b})^2,&0\leq\nu\leq\frac{1}{2}\end{array}\right.
\end{equation}
where $r_0=\min\{2r,1-2r\}$ for $r=\min\{\nu,1-\nu\}.$ These inequalities refine the reversed version  $\nu a+(1-\nu)b\leq a^{\nu}b^{1-\nu}+\max\{\nu,1-\nu\}(\sqrt{a}-\sqrt{b})^2$; cf. \cite{kittanehmanasreh2}.

Moreover, it is proved in \cite{zhao} that
\begin{equation}\label{reverse_zhao_square}
\left\{\begin{array}{cc}(\nu a+(1-\nu)b)^2+r_0(\sqrt{ab}-a)^2\leq \left(a^{\nu}b^{1-\nu}\right)^2+(1-\nu)^2(a-b)^2,&0\leq\nu\leq\frac{1}{2}\\
(\nu a+(1-\nu)b)^2+r_0(\sqrt{ab}-b)^2\leq \left(a^{\nu}b^{1-\nu}\right)^2+\nu^2(a-b)^2,&0\leq\nu\leq\frac{1}{2}\end{array}\right.
\end{equation}
refining the squared version $(\nu a+(1-\nu)b)^2\leq \left(a^{\nu}b^{1-\nu}\right)^2+\max\{\nu,1-\nu\}^2(a-b)^2$ of \cite{kittanehmanasreh2}.

Throughout the paper, $\mathbb{M}_n$ denotes the space of all $n\times n$ complex matrices, $\mathbb{M}_n^{+}$ is the cone of positive semidefinite matrices in $\mathbb{M}_n$, $\mathbb{M}_n^{++}$ is the cone of strictly positive definite matrices in $\mathbb{M}_n$ and $\||\;\;\||$ is a unitarily invariant norm defined on $\mathbb{M}_n$.

Recall that the notation $X\geq Y$ for $X,Y\in \mathbb{M}_n$ means that $X$ and $Y$ are hermitian and $X-Y\in\mathbb{M}_n^{+}.$ Moreover, for two  operators $A,B\in\mathbb{M}_n^{++}$ and $\nu\in\mathbb{R},$ we use the notations
\begin{eqnarray*}
A\nabla_{\nu}B=(1-\nu) A+\nu B\;{\text{and}}\; A\#_\nu B=A^{\frac{1}{2}}\left(A^{-\frac{1}{2}}BA^{-\frac{1}{2}}\right)^{\nu}A^{\frac{1}{2}}.
\end{eqnarray*}
The above refinements for real numbers are extended to matrices in different approaches.

Our main goal in this article is to give a full description of these refinements and reverses. In particular, we utilize our recent refinement in \cite{sabchoiref} to show other multiple-term refinements of the original Young's inequality and its reverses. Then we use these refinements to obtain improved Heinz inequalities for operators.

To state our results, we adopt the following notations. Let $a,b>0$ and $\nu\in [0,1]$. For $N\in\mathbb{N}$ and $j=1,2,\cdots,N$,  let $$k_j(\nu)=[2^{j-1}\nu], r_j(\nu)=[2^{j}\nu]\;{\text{and}}\;s_j(\nu)=(-1)^{r_j(\nu)}2^{j-1}\nu+(-1)^{r_j(\nu)+1}\left[\frac{r_j(\nu)+1}{2}\right].$$
Then define the positive function
\begin{equation}\label{SN_definition}
S_N(\nu;a,b)=\sum_{j=1}^{N}s_j(\nu)
\left(\sqrt[2^j]{b^{2^{j-1}-k_j(\nu)}a^{k_j(\nu)}}-\sqrt[2^{j}]{a^{k_j(\nu)+1}b^{2^{j-1}-k_j(\nu)-1}}\right)^2.
\end{equation}

In \cite{sabchoiref}, it is proved that
\begin{equation}\label{the_first_refinement}
a^{\nu}b^{1-\nu}+S_N(\nu;a,b)\leq \nu a+(1-\nu)b.
\end{equation}
and
\begin{equation}\label{exact_difference}
\nu a+(1-\nu)b-S_N(\nu;a,b)=R_N(\nu;a,b)
\end{equation}
where
\begin{eqnarray}
\nonumber&&R_N(\nu;a,b)\\
\nonumber&=&\left([2^{N}\nu]+1-2^{N}\nu\right)\sqrt[2^{N}]{a^{[2^N\nu]}b^{2^{N}-[2^{N}\nu]}}+\left(2^{N}\nu-[2^{N}\nu]\right)\sqrt[2^{N}]
{a^{[2^N\nu]+1}b^{2^{N}-[2^{N}\nu]-1}}.\\
\label{Rformula}&&
\end{eqnarray}
In the beginning, we prove the following property of $S_N$ and $R_N$. This property helps us to write the refining term in an appropriate way.
\begin{proposition}
For $a,b>0, 0\leq \nu\leq 1, N\in\mathbb{N}$ and $S_N$ and $R_N$ as above. We have
$S_N(\nu;a,b)=S_N(1-\nu;b,a)$ and $R_N(\nu;a,b)=R_N(1-\nu;b,a).$
\end{proposition}
\begin{proof}
We show that $R_N(\nu;a,b)=R_N(1-\nu;b,a),$ which then leads to other conclusion via (\ref{exact_difference}). We obtain $R(1-\nu;b,a)$ by replacing $\nu, a$ and $b$ by $1-\nu,b$ and $a$, respectively, in (\ref{Rformula}). Therefore
\begin{eqnarray*}
R_N(1-\nu;b,a)&=&\left([2^{N}(1-\nu)]+1-2^{N}(1-\nu)\right)\sqrt[2^{N}]{b^{[2^{N}(1-\nu)]}a^{2^{N}-[2^{N}(1-\nu)]}}\\
&&+\left(2^{N}(1-\nu)-[2^{N}(1-\nu)]\right)\sqrt[2^{N}]{b^{[2^{N}(1-\nu)]+1}a^{2^{N}-[2^{N}(1-\nu)]-1}}\\
&=&(1+2^{N}\nu+[-2^{N}\nu])\sqrt[2^{N}]{b^{2^{N}+[-2^{N}\nu]}a^{-[-2^{N}\nu]}}+\\
&&+(-2^{N}\nu-[-2^{N}\nu])\sqrt[2^{N}]{b^{2^{N}+[-2^{N}]+1}a^{-[-2^{N}\nu]-1}}.
\end{eqnarray*}
To simplify this expression, we have two cases. The first is when $[2^{N}\nu]$ is an integer. Considering this in the above equation gives
 $R_N(1-\nu;b,a)=\sqrt[2^{N}]{a^{2^{N}\nu}b^{2^{N}-2^{N}\nu}}.$ Computing $R_N(\nu;a,b)$ when $[2^{N}\nu]$ is an integer gives $R_N(\nu;a,b)=\sqrt[2^{N}]{a^{2^{N}\nu}b^{2^{N}-2^{N}\nu}}.$ This completes the proof for this case. Now if $[2^{N}\nu]$ is not an integer, then $[-2^{N}\nu]=-[2^{N}\nu]-1.$ Substituting this back in $R_{N}(1-\nu;b,a)$ gives $R_N(\nu;a,b).$
\end{proof}
In establishing the reversed version, we need the inequality
\begin{equation}\label{heinz}
2\sqrt{ab}\leq a^{\nu}b^{1-\nu}+a^{1-\nu}b^{\nu}\leq a+b
\end{equation}
known as the Heinz mean inequality; see \cite{BM, KAU} and references therein for more information on Heinz inequality.\\
The organization of the paper is as follows. In the first part, we present refinements of the reversed Young's inequality, then a refinement of the squared version with its reverse is presented. After that, further results generalizing several inequalities in the literature are proved. Next, we present our refinements in terms of the Kantotovich constant, which has been used extensively in recent refinements. Interestingly, these refinements are then applied to log-convex functions and operators. In the last section we present several Young and Heinz type inequalities for the Hilbert-Schmidt norm, improving numerous results in the literature.\\
We emphasize that the significance of the results in this article is to have as many refining terms as we wish, unlike all other refinements in the literature where at most two refining terms have been found.

\section{Numerical versions}
In this part of the paper, we present the numerical inequalities needed to prove the operator versions.

\subsection{ Refined Reverses}
We prove first refined reverses of Young's inequality, generalizing (\ref{reverse_zhao}) and other inequalities in the literature. This is accomplished by adding $N$ refining terms, instead of one term as in (\ref{reverse_zhao}).
\begin{theorem}\label{our_first_reverse}
Let $a,b>0$ and $N\in\mathbb{N}.$ If $0\leq\nu\leq \frac{1}{2}$, then
$$\nu a+(1-\nu) b+S_N(2\nu;\sqrt{ab},a)\leq a^{\nu}b^{1-\nu}+(1-\nu)(\sqrt{a}-\sqrt{b})^2.$$
On the other hand, if $\frac{1}{2}\leq \nu\leq 1$, then
$$\nu a+(1-\nu) b+S_N(2-2\nu;\sqrt{ab},b)\leq a^{\nu}b^{1-\nu}+\nu(\sqrt{a}-\sqrt{b})^2.$$
\end{theorem}
\begin{proof}
If $0\leq \nu\leq\frac{1}{2}$, then
\begin{eqnarray*}
&&\hspace{-2cm}a^{\nu}b^{1-\nu}+(1-\nu)(\sqrt{a}-\sqrt{b})^2-\left(\nu a+(1-\nu) b\right)\\
&=&a^{\nu}b^{1-\nu}+(1-2\nu)a+2\nu \sqrt{ab}-2\sqrt{ab}\\
&\geq&a^{\nu}b^{1-\nu}+a^{1-2\nu}(\sqrt{ab})^{2\nu}+S_N(2\nu;\sqrt{ab},a)-2\sqrt{ab}\;\;\;(\text{by}\;\ref{the_first_refinement})\\
&=&\left(a^{\nu}b^{1-\nu}+a^{1-\nu}b^{\nu}-2\sqrt{ab}\right)+S_N(2\nu;\sqrt{ab},a)\\
&\geq&S_N(2\nu;\sqrt{ab},a),
\end{eqnarray*}
where we have used (\ref{heinz}) to get the last inequality. Thus, we have proved
$$\nu a+(1-\nu) b+S_N(2\nu;\sqrt{ab},a)\leq a^{\nu}b^{1-\nu}+(1-\nu)(\sqrt{a}-\sqrt{b})^2,\;0\leq \nu\leq \frac{1}{2}.$$
Similar computations give
$$\nu a+(1-\nu) b+S_N(2-2\nu;\sqrt{ab},b)\leq a^{\nu}b^{1-\nu}+\nu(\sqrt{a}-\sqrt{b})^2,\;\frac{1}{2}\leq \nu\leq 1.$$
\end{proof}
To make this inequality better understood in contrast to (\ref{reverse_zhao}), we find the $S_1$ term appearing in Theorem \ref{our_first_reverse}.\\
For $0\leq \nu<\frac{1}{2}$, direct computations show that $k_1(2\nu)=0$ and
{\small $$r_1(2\nu)=\left\{\begin{array}{cc}2\nu,&0\leq \nu<\frac{1}{4}\\ 1-2\nu,&\frac{1}{4}\leq \nu<\frac{1}{2}\end{array}\right., S_1(2\nu;\sqrt{ab},a)=\left\{\begin{array}{cc}2\nu(\sqrt[4]{ab}-\sqrt{a})^2,&0\leq\nu<\frac{1}{4}\\
(1-2\nu)(\sqrt[4]{ab}-\sqrt{a})^2,&\frac{1}{4}\leq\nu<\frac{1}{2}\end{array}\right..$$} Thus, for $0\leq\nu<\frac{1}{2}$ we have
$S_1(2\nu;\sqrt{ab},a)=\min\{2\nu,1-2\nu\}(\sqrt[4]{ab}-\sqrt{a})^2,$ which is exactly the refining term appearing in (\ref{reverse_zhao}). Similar computations show that $S_1(2-2\nu;\sqrt{ab},b)$ is the refining term in (\ref{reverse_zhao}) for $\frac{1}{2}<\nu<1.$\\
This justifies why our refinement in Theorem \ref{our_first_reverse} is much better than (\ref{reverse_zhao}).

For $\nu>0$ we have the reversed version of Young's inequality
\begin{equation}\label{minusreverse}
(1+\nu)a-\nu b\leq a^{1+\nu}b^{-\nu}.
\end{equation}
This inequality becomes handy in proving reversed versions of Young's inequality. The following refined version has been recently proved in \cite{mojtaba}

\begin{equation}\label{refinedminus}
(1+\nu)a-\nu b+\nu(\sqrt{a}-\sqrt{b})^2\leq a^{1+\nu}b^{-\nu}, a,b,\nu>0.
\end{equation}

The following is a refinement of this inequality.
\begin{theorem}\label{theorem_reverse_mojtaba}
Let $a,b>0$ and $\nu\geq 0$. Then
\begin{equation}
(1+\nu)a-\nu b+\nu\sum_{j=1}^{N}2^{j-1}\left(\sqrt{a}-\sqrt[2^j]{a^{2^{j-1}-1}b}\right)^2\leq a^{1+\nu}b^{-\nu}
\end{equation}
for each $N\in\mathbb{N}$.
\end{theorem}
\begin{proof}
We prove first, by induction, that
\begin{equation}\label{wanted_induction_reverse}
(1+\nu)a-\nu b+\nu\sum_{j=1}^{N}2^{j-1}\left(\sqrt{a}-\sqrt[2^j]{a^{2^{j-1}-1}b}\right)^2=(1+2^{N}\nu)a-2^{N}\nu \sqrt[2^{N}]{a^{2^{N}-1}b}.
\end{equation}
When $N=1,$ this is (\ref{refinedminus}). Assume (\ref{wanted_induction_reverse}) is true for some $N\in\mathbb{N}$. Then,
\begin{eqnarray*}
(1+\nu)a&-&\nu b+\nu\sum_{j=1}^{N+1}2^{j-1}\left(\sqrt{a}-\sqrt[2^j]{a^{2^{j-1}-1}b}\right)^2\\
&=&(1+\nu)a-\nu b+\nu\sum_{j=1}^{N}2^{j-1}\left(\sqrt{a}-\sqrt[2^j]{a^{2^{j-1}-1}b}\right)^2+\\
&&+2^{N}\nu\left(\sqrt{a}-\sqrt[2^{N+1}]{a^{2^{N}-1}b}\right)^2\;({\text{apply\;the\;inductive\;step\;now}})\\
&=&(1+2^{N}\nu)a-2^{N}\nu \sqrt[2^{N}]{a^{2^{N}-1}b}+2^{N}\nu\left(a+\sqrt[2^N]{a^{2^N-1}b}-2\sqrt[2^{N+1}]{a^{2^{N}-1}b}\right)\\
&=&(1+2^{N+1}\nu)a-2^{N}\nu \sqrt[2^{N+1}]{a^{2^{N+1}-1}b}.
\end{eqnarray*}
This proves (\ref{wanted_induction_reverse}). Now apply (\ref{minusreverse}) on (\ref{wanted_induction_reverse}) to get the result.
\end{proof}

Alternatively, we may apply Theorem \ref{our_first_reverse} to obtain other refinements of (\ref{minusreverse}).
\begin{corollary}\label{reverse_after_mojtaba}
Let $a,b,>0$ and $N\in\mathbb{N}$. Then
\begin{eqnarray}\label{completerefinedminusfirst}
(1+\nu)a-\nu b+\frac{1}{b}S_N(1-\nu;ab,b^2)\leq a^{\nu+1}b^{-\nu}
\end{eqnarray}
for $0\leq \nu\leq 1$.
\end{corollary}
\begin{proof}
For $0\leq \nu\leq 1,$ let $t=\frac{\nu+1}{2}$. Then, $\frac{1}{2}\leq t\leq 1$ and Theorem \ref{our_first_reverse} ensure that
$$t a+(1-t) b+S_N(2-2t;\sqrt{ab},b)\leq a^{t}b^{1-t}+t(\sqrt{a}-\sqrt{b})^2$$
or
$$t a+(1-t) b-t(\sqrt{a}-\sqrt{b})^2+S_N(2-2t;\sqrt{ab},b)\leq a^{t}b^{1-t}.$$ Simplifying this expression and replacing $t$ by $\frac{\nu+1}{2}$ we reach
$$(1+\nu)\sqrt{ab}-\nu b+S_N(1-\nu;\sqrt{ab},b)\leq a^{\frac{\nu+1}{2}}b^{\frac{1-\nu}{2}}.$$ Replacing $a$ by $a^2$, $b$ by $b^2$ and dividing by $b$ imply the required inequality.
\end{proof}

Further, we have the refinement.
\begin{corollary}
Let $a,b>0,\nu>0$ and $N\in\mathbb{N}$. Then
\begin{eqnarray*}
(1+\nu)a-\nu b+(1+\nu)S_N\left(\frac{1}{1+\nu};a^{1+\nu}b^{-\nu},b\right)\leq a^{1+\nu}b^{-\nu}.
\end{eqnarray*}
\end{corollary}
\begin{proof}
For such $a,b$ and $\nu$, let $x=a^{1+\nu}b^{-\nu}, y=b$ and $t=\frac{1}{\nu+1}.$ By (\ref{the_first_refinement}), we have $x^ty^{1-t}+S_N(t;x,y)\leq tx+(1-t)y.$ That is
$$a+S_N\left(\frac{1}{1+\nu};a^{1+\nu}b^{-\nu},b\right)\leq \frac{1}{\nu+1}a^{1+\nu}b^{-\nu}+\frac{\nu}{1+\nu}b,$$ which is equivalent to
$$(1+\nu)a-\nu b+(1+\nu)S_N\left(\frac{1}{1+\nu};a^{1+\nu}b^{-\nu},b\right)\leq a^{1+\nu}b^{-\nu}.$$
\end{proof}

\subsection{The Squared Version}
In this part of the paper, we prove a refinement of the squared versions in \cite{omarkittaneh} and \cite{zhao}.
\begin{theorem}\label{refined_square}
Let $a,b>0$ and $\nu\in [0,1].$ Given $N\in \mathbb{N}, N\geq 2$, let $r_j(\nu), k_j(\nu)$ and $s_j(\nu)$ be as above. Then it holds that
\begin{eqnarray}
\nonumber&&\hspace{-2cm}\left(a^{\nu}b^{1-\nu}\right)^2+s_1^2(\nu)(a-b)^2\\
\nonumber&&+\sum_{j=2}^{N}s_j(\nu)
\left(\sqrt[2^{j-1}]{b^{2^{j-1}-k_j(\nu)}a^{k_j(\nu)}}-\sqrt[2^{j-1}]{a^{k_j(\nu)+1}b^{2^{j-1}-k_j(\nu)-1}}\right)^2\\
&\leq& \label{squared_inequality_first}\left(\nu a+(1-\nu)b\right)^2.
\end{eqnarray}
\end{theorem}
\begin{proof}
For such $a,b$ and $\nu$ it follows from (\ref{the_first_refinement}) that
\begin{eqnarray}
\nonumber&&\hspace{-2cm}a^{\nu}b^{1-\nu}+s_1(\nu)(\sqrt{a}-\sqrt{b})^2\\
\nonumber&&+\sum_{j=2}^{N}s_j(\nu)
\left(\sqrt[2^j]{b^{2^{j-1}-k_j(\nu)}a^{k_j(\nu)}}-\sqrt[2^{j}]{a^{k_j(\nu)+1}b^{2^{j-1}-k_j(\nu)-1}}\right)^2\\
\label{needed_first}&\leq& \nu a+(1-\nu)b.
\end{eqnarray}
In (\ref{needed_first}), replace $a$ by $a^2$ and $b$ by $b^2$ then add $s_1^2(\nu)(a-b)^2$ to both sides and add and subtract $(v a+(1-v)b)^2$ to the right side to get
\begin{eqnarray*}
\nonumber&&\hspace{-2cm}\left(a^{\nu}b^{1-\nu}\right)^2+s_1^2(\nu)(a-b)^2\\
\nonumber&&+\sum_{j=2}^{N}s_j(\nu)
\left(\sqrt[2^{j-1}]{b^{2^{j-1}-k_j(\nu)}a^{k_j(\nu)}}-\sqrt[2^{j-1}]{a^{k_j(\nu)+1}b^{2^{j-1}-k_j(\nu)-1}}\right)^2\\
&\leq& \left(\nu a+(1-\nu)b\right)^2\\
&&+(v a^2+(1-v)b^2)+s_1^2(a-b)^2-s_1(a-b)^2-(v a+(1-v)b)^2.
\end{eqnarray*}
Then the result follows by noting that $$(v a^2+(1-v)b^2)+s_1^2(\nu)(a-b)^2-s_1(\nu)(a-b)^2-(v a+(1-v)b)^2=0.$$
\end{proof}

Now replacing $a$ by $a^2$ and $b$ by $b^2$ in Theorem \ref{our_first_reverse}, and noting that $(\nu a+(1-\nu)b)^2- (\nu a^2+(1-\nu)b^2)+\max\{\nu,1-\nu\} (a - b)^2 - \max\{\nu,1-\nu\}^2 (a - b)^2=0,$ we obtain the following reverse of Theorem \ref{refined_square}, which is a refinement of (\ref{reverse_zhao_square}).
\begin{proposition}\label{reversed_square_reversed}
Let $a,b>0$ and $N\in\mathbb{N}$. If $0\leq\nu\leq \frac{1}{2}$, then
$$(\nu a+(1-\nu) b)^2+S_N(2\nu;ab,a^2)\leq \left(a^{\nu}b^{1-\nu}\right)^2+(1-\nu)^2(a-b)^2.$$
On the other hand, if $\frac{1}{2}\leq \nu\leq 1$, then
$$(\nu a+(1-\nu) b)^2+S_N(2-2\nu;ab,b^2)\leq \left(a^{\nu}b^{1-\nu}\right)^2+\nu^2(a-b)^2.$$
\end{proposition}
\subsection{Remarks on Refinements}
Refining Young's inequality can go forever, in the sense that we can find further refining terms for all refinements in the literature. This is done by successively applying the same or other refinements.\\

\begin{proposition}\label{double_refinement}
Let $a,b>0, N,M\in \mathbb{N}$ and $0\leq\nu\leq 1$. Then
$$a^{\nu}b^{1-\nu}+S_N(\nu;a,b)+S_M(\alpha;x,y)\leq\nu a+(1-\nu)b,$$
where $\alpha=[2^{N}\nu]+1-2^{N}\nu, x=\sqrt[2^{N}]{a^{[2^N\nu]}b^{2^{N}-[2^{N}\nu]}}$ and $y=\sqrt[2^{N}]
{a^{[2^N\nu]+1}b^{2^{N}-[2^{N}\nu]-1}}.$
\end{proposition}
\begin{proof}
Recalling (\ref{exact_difference}) and noting (\ref{the_first_refinement}), we have
\begin{eqnarray*}
\nu a+(1-\nu)b-S_N(\nu;a,b)&=&\alpha x+(1-\alpha)y\\
&\geq& x^{\alpha}y^{1-\alpha}+S_M(\alpha;x,y)\\
&=&a^{\nu}b^{1-\nu}+S_M(\alpha;x,y).
\end{eqnarray*}
This yields that
$$a^{\nu}b^{1-\nu}+S_N(\nu;a,b)+S_M(\alpha;x,y)\leq\nu a+(1-\nu)b.$$
\end{proof}

Following the same reasoning as in Theorem \ref{refined_square}, we obtain the following refinement of the squared version.
\begin{theorem}
Let $a,b>0, N,M\in \mathbb{N}$ and $0\leq\nu\leq 1$. Then
\begin{align*}
& \left(a^{\nu}b^{1-\nu}\right)^2+s_1^2(\nu)(a-b)^2+S_M(\alpha;x,y)\\
&+\sum_{j=2}^{N}s_j(\nu)
\left(\sqrt[2^{j-1}]{b^{2^{j-1}-k_j(\nu)}a^{k_j(\nu)}}-\sqrt[2^{j-1}]{a^{k_j(\nu)+1}b^{2^{j-1}-k_j(\nu)-1}}\right)^2\leq(\nu a+(1-\nu)b)^2,
\end{align*}
where $\alpha=[2^{N}\nu]+1-2^{N}\nu, x=\sqrt[2^{N}]{a^{2[2^N\nu]}b^{2^{N+1}-2[2^{N}\nu]}}$ and $y=\sqrt[2^{N}]
{a^{2[2^N\nu]+2}b^{2^{N+1}-2[2^{N}\nu]-2}}.$
\end{theorem}
Applying the same reasoning gives refined reverses, as well. Before stating the reversed version, let $\alpha_N(\nu)=[2^{N}\nu]+1-2^{N}\nu,$ $x_N(\nu;a,b)=\sqrt[2^{N}]{a^{[2^N\nu]}b^{2^{N}-[2^{N}\nu]}}$ and $y_N(\nu;a,b)=\sqrt[2^{N}]
{a^{[2^N\nu]+1}b^{2^{N}-[2^{N}\nu]-1}}.$ Then the reversed versions can be stated in terms of these quantities as follows.
\begin{theorem}
Let $a,b>0, N,M\in\mathbb{N}$. If $0\leq\nu\leq\frac{1}{2}$, then
\begin{eqnarray*}
\nu a+(1-\nu)b+S_N(2\nu;\sqrt{ab},a)&+&S_M\left(\alpha_N(2\nu);x_N(2\nu;\sqrt{ab},a),y_N(2\nu;\sqrt{ab},a)\right)\\
&\leq& a^{\nu}b^{1-\nu}+(1-\nu)(\sqrt{a}-\sqrt{b})^2,
\end{eqnarray*}
On the other hand, if $\frac{1}{2}\leq\nu\leq 1$, then
\begin{eqnarray*}
\nu a+(1-\nu)b&+&S_N(2-2\nu;\sqrt{ab},b)\\
&+&S_M\left(\alpha_N(2-2\nu);x_N(2-2\nu;\sqrt{ab},b),y_N(2-2\nu;\sqrt{ab},b)\right)\\
&\leq& a^{\nu}b^{1-\nu}+\nu(\sqrt{a}-\sqrt{b})^2,
\end{eqnarray*}
\end{theorem}
\begin{proof}
For $0\leq\nu\leq\frac{1}{2},$ we have, by Theorem \ref{double_refinement},
\begin{eqnarray*}
a^{\nu}b^{1-\nu}&+&(1-\nu)(\sqrt{a}-\sqrt{b})^2-\left(\nu a+(1-\nu)b\right)\\
&=&a^{\nu}b^{1-\nu}+(1-2\nu)a+2\nu \sqrt{ab}-2\sqrt{ab}\\
&\geq&a^{\nu}b^{1-\nu}+a^{1-2\nu}\sqrt{ab}^{2\nu}-2\sqrt{ab}+\\
&&+S_N(2\nu;\sqrt{ab},a)+S_M\left(\alpha_N(2\nu);x_N(2\nu;\sqrt{ab},a),y_N(2\nu;\sqrt{ab},a)\right)\\
&\geq&S_N(2\nu;\sqrt{ab},a)+S_M\left(\alpha_N(2\nu);x_N(2\nu;\sqrt{ab},a),y_N(2\nu;\sqrt{ab},a)\right).
\end{eqnarray*}
Observe that the formulae are obtained by the formulae of Theorem \ref{double_refinement} on replacing $\nu,a$ and $b$ by $2\nu,\sqrt{ab}$ and $a,$ respectively.\\
On the other hand, if $\frac{1}{2}\leq\nu\leq 1,$ we have
\begin{eqnarray*}
a^{\nu}b^{1-\nu}&+&\nu(\sqrt{a}-\sqrt{b})^2-\left(\nu a+(1-\nu)b\right)\\
&=&a^{\nu}b^{1-\nu}+(2\nu-1)b+(2-2\nu)\sqrt{ab}-2\sqrt{ab}\\
&\geq&a^{\nu}b^{1-\nu}+a^{1-\nu}b^{\nu}-2\sqrt{ab}+S_N(2-2\nu,\sqrt{ab},b)+\\
&&+S_M\left(\alpha_N(2-2\nu),x_N(2-2\nu;\sqrt{ab},b),y_N(2-2\nu,\sqrt{ab},b)\right)\\
&\geq&S_N(2-2\nu,\sqrt{ab},b)\\
&&+S_M\left(\alpha_N(2-2\nu),x_N(2-2\nu;\sqrt{ab},b),y_N(2-2\nu,\sqrt{ab},b)\right).
\end{eqnarray*}
\end{proof}

\subsection{The Kantorovich Constant}
Recent refinements of Young's inequality have deployed the Kantorovich constant $K(t,2):=\frac{(t+1)^2}{4t}, t>0.$ For example, it is proved in \cite{zuo} that
\begin{equation}\label{kanto_young}
K(h,2)^{r}a^{\nu}b^{1-\nu}\leq\nu a+(1-\nu)b, 0\leq\nu\leq 1,\;r=\min\{\nu,1-\nu\}, h=\frac{b}{a}.
\end{equation}
What makes this inequality a refinement of (\ref{original_young}) is the fact that $K(t,2)\geq 1$ when $t>0.$ Some refinements of this inequality have been given in the literature, but again with one refining term. For example, it is shown in \cite{wu} that for $r=\min\{\nu,1-\nu\}$ one has
\begin{equation}\label{kanto_ref}
K(\sqrt{h},2)^{r'}a^{\nu}b^{1-\nu}+r(\sqrt{a}-\sqrt{b})^2\leq \nu a+(1-\nu)b, r'=\min\{2r,1-2r\},h=\frac{b}{a}.
\end{equation}
 To better state our next result, which is a refinement of (\ref{kanto_ref}), we use the notations $\alpha_N(\nu)=1+[2^{N}\nu]-2^{N}\nu$ and $\beta_N(\nu)=\min\{\alpha_N(\nu),1-\alpha_N(\nu)\}$ for $N\in\mathbb{N}$.

\begin{theorem}
Let $a,b>0,0\leq\nu\leq 1$ and $N\in\mathbb{N}$. Then
\begin{equation}\label{our_ref_kanto}
K\left(\sqrt[2^N]{\frac{b}{a}},2\right)^{\beta_N(\nu)}a^{\nu}b^{1-\nu}+S_N(\nu;a,b)\leq \nu a+(1-\nu)b.
\end{equation}
\end{theorem}
\begin{proof}
By (\ref{exact_difference}) and (\ref{Rformula}) we have
\begin{eqnarray*}
\nu a&+&(1-\nu)b-S_N(\nu;a,b)=R_N(\nu;a,b)\\
&=&\alpha_N(\nu)\sqrt[2^{N}]{a^{[2^N\nu]}b^{2^{N}-[2^{N}\nu]}}+(1-\alpha_N(\nu))\sqrt[2^{N}]
{a^{[2^N\nu]+1}b^{2^{N}-[2^{N}\nu]-1}}\\
&\geq&K\left(\frac{\sqrt[2^{N}]
{a^{[2^N\nu]+1}b^{2^{N}-[2^{N}\nu]-1}}}{\sqrt[2^{N}]{a^{[2^N\nu]}b^{2^{N}-[2^{N}\nu]}}},2\right)^{\beta_N(\nu)}\times
\left(\sqrt[2^{N}]{a^{[2^N\nu]}b^{2^{N}-[2^{N}\nu]}}\right)^{\alpha_N(\nu)}\\
&\times&\left(\sqrt[2^{N}]
{a^{[2^N\nu]+1}b^{2^{N}-[2^{N}\nu]-1}}\right)^{1-\alpha_N(\nu)}\;({\text{by}}\;(\ref{kanto_young}))\\
&=&K\left(\sqrt[2^N]{\frac{b}{a}},2\right)^{\beta_N(\nu)}a^{\nu}b^{1-\nu}.
\end{eqnarray*}
\end{proof}

Notice that when $N=1$, (\ref{our_ref_kanto}) gives (\ref{kanto_ref}). On the other hand, when $N=0$, we get (\ref{kanto_young}), where by convention $S_0(\nu;a,b)=0.$

In \cite{sabLM}, the following version of Young's inequality was proved.
$$|1-2\nu|^{2r_0}(a^{\nu}b^{1-\nu})^2+r_0^2(a+b)^2\leq (\nu a+(1-\nu)b)^2,\;{\text{where}}\;r_0=\min\{\nu,1-\nu\}.$$
On the other hand, it is proved in \cite{xing} that
$$r_0^{2r_0}a^{\nu}b^{1-\nu}+r_0(\sqrt{a}-\sqrt{b})^2\leq \nu^2a+(1-\nu)^2b\;{\text{where}}\;r_0=\min\{\nu,1-\nu\}.$$
Now using (\ref{our_ref_kanto}), we easily get the following refined versions of these inequalities, invoking the Kantorovich constant,
\begin{theorem}
Let $a,b\geq 0$ and $N\in\mathbb{N}$. If $0\leq \nu\leq \frac{1}{2}$ then
\begin{eqnarray*}
&&(1-2\nu)^{2\nu}K\left(\sqrt[2^N]{\frac{b}{(1-2\nu)a}},2\right)^{\beta_N(2\nu)}(a^{\nu}b^{1-\nu})^2+\nu^2(a+b)^2\\
&+&(1-2\nu)bS_N\left(1-2\nu;\frac{b}{1-2\nu},a\right)\leq (\nu a+(1-\nu)b)^2.
\end{eqnarray*}
If $\frac{1}{2}\leq\nu\leq 1$ we have
\begin{eqnarray*}
&&(2\nu-1)^{2-2\nu}K\left(\sqrt[2^N]{\frac{(2\nu-1)b}{a}},2\right)^{\beta_N(2\nu-1)}(a^{\nu}b^{1-\nu})^2\\
&+&(1-\nu)^2(a+b)^2+(2\nu-1)aS_N\left(2\nu-1;\frac{a}{2\nu-1},b\right)\leq (\nu a+(1-\nu)b)^2.
\end{eqnarray*}
\end{theorem}

\begin{theorem}
Let $a,b>0$ and $N\in\mathbb{N}$. If $0\leq \nu\leq \frac{1}{2}$ then
\begin{eqnarray*}
\nu^{2\nu}K\left(\sqrt[2^N]{\nu\sqrt{\frac{a}{b}}},2\right)^{\beta_N(1-2\nu)}a^{\nu}b^{1-\nu}&+&\nu^2(\sqrt{a}-\sqrt{b})^2+S_N\left(1-2\nu;b,\nu\sqrt{ab}\right)\\
&\leq& \nu^2a+(1-\nu)^2b.
\end{eqnarray*}
If $\frac{1}{2}\leq \nu\leq 1,$ then
\begin{eqnarray*}
&&\hspace{-2cm}(1-\nu)^{2-2\nu}K\left(\sqrt[2^N]{(1-\nu)\sqrt{\frac{b}{a}}},2\right)^{\beta_N(2\nu-1)}a^{\nu}b^{1-\nu}+(1-\nu)^2(\sqrt{a}-\sqrt{b})^2\\
&+&S_N\left(2\nu-1;a,(1-\nu)\sqrt{ab}\right)\leq \nu^2 a+(1-\nu)^2b.
\end{eqnarray*}
\end{theorem}

\section{Convexity Approach}
Recall that a function $f:I\to \mathbb{R}^+$ is said to be log-convex if $f(\lambda t_1+(1-\lambda)t_2)\leq f^{\lambda}(t_1)f^{1-\lambda}(t_2)$ for $t_1,t_2\in I$ and $0\leq \lambda\leq 1.$ Young-type inequalities seem to be strongly related to log-convex functions. A good reference about this relation is the recent work \cite{saboam}.

\begin{proposition}\label{log_convex_first}
Let $f:[0,1]\to\mathbb{R}^{+}$ be log-convex and let $N\in\mathbb{N}$. Then
$$K\left(\sqrt[2^N]{\frac{f(1)}{f(0)}},2\right)^{\beta_N(t)}f(t)+S_{N}(t;f(1),f(0))\leq t f(1)+(1-t)f(0).$$
\end{proposition}
\begin{proof}
Note that log-convexity of $f$ implies
\begin{eqnarray*}
f(t)&=&f(t\cdot 1+(1-t)\cdot 0)\leq f^{t}(1)f^{1-t}(0).
\end{eqnarray*}
Consequently,
\begin{eqnarray*}
&&\hspace{-2cm}K\left(\sqrt[2^N]{\frac{f(1)}{f(0)}},2\right)^{\beta_N(t)}f(t)+S_{N}(t;f(1),f(0))\\
&\leq&K\left(\sqrt[2^N]{\frac{f(1)}{f(0)}},2\right)^{\beta_N(t)}f^{t}(1)f^{1-t}(0)+S_{N}(t;f(1),f(0))\\
&\leq& t f(1)+(1-t)f(0),
\end{eqnarray*}
where (\ref{our_ref_kanto}) has been used to obtain the last line.
\end{proof}
Thus, these refined Young-type inequalities can be applied to any log-convex function. In \cite{saboam} it has been proved that the functions
$$f_1(t)=\||A^{t}XB^{t}\||, f_2(t)=\||A^{t}XB^{1-t}\||, f_3(t)=\||A^{t}\||\;{\text{and}}\; f_4(t)=\tr(A^{t}XB^{1-t}X^*)$$ are log-convex on $[0,1].$ In this context $A,B\in \mathbb{M}_n^{+}$ and $X\in \mathbb{M}_n$. Employing Proposition \ref{log_convex_first} to these functions we reach the next result.
\begin{corollary}\label{corollary_for_operators}
Let $N\in\mathbb{N},A,B\in\mathbb{M}_n^{+}$, $X\in\mathbb{M}_n$, $0\leq t\leq 1$ and $\||\;\;\||$ be any unitarily invariant norm. Then
\begin{enumerate}
\item \begin{eqnarray*}
&&\hspace{-2cm}K\left(\sqrt[2^N]{\frac{\||AXB\||}{\||X\||}},2\right)^{\beta_N(t)}\||A^{t}XB^{t}\||+S_N(t;\||AXB\||,\||X\||)\\
&\leq& t\||AXB\||+(1-t)\||X\||;\\
\end{eqnarray*}
\item \begin{eqnarray*}
&&\hspace{-2cm}K\left(\sqrt[2^N]{\frac{\||AX\||}{\||XB\||}},2\right)^{\beta_N(t)}\||A^{t}XB^{1-t}\||+S_N(t;\||AX\||,\||XB\||)\\
&\leq& t\||AX\||+(1-t)\||XB\||;\\
\end{eqnarray*}
\item \begin{eqnarray*}
K\left(\sqrt[2^N]{\frac{\||A\||}{\||I\||}},2\right)^{\beta_N(t)}\||A^{t}\||+S_N(t;\||A\||,\||I\||)\leq t\||A\||+(1-t)\;\||I\||;\\
\end{eqnarray*}
\item \begin{eqnarray*}
&&\hspace{-1cm}K\left(\sqrt[2^N]{\frac{\tr(A|X|^2)}{\tr(B|X^*|^2)}},2\right)^{\beta_N(t)}\tr(A^{t}XB^{1-t}X^*)+S_N(t;\tr(A|X|^2),\tr(B|X^*|^2)\\
&\leq& t\;\tr(A|X|^2)+(1-t)\tr(B|X^*|^2).
\end{eqnarray*}
\end{enumerate}
\end{corollary}
These are operator versions, invoking unitarily invariant norms. In the next part of the paper, we present some operator results without using any norm.

Refinements of log-convex-type inequalities do not stop here! In \cite{saboam} it is proved that when $f:[0,1]\to\mathbb{R}^+$ is log-convex, we have for $N\in\mathbb{N}$,
\begin{equation}\label{ineq_for_logconv_from_oam}
f(\nu)\leq f\left(\frac{[2^{N}\nu]}{2^{N}}\right)^{[2^{N}\nu]+1-2^{N}\nu}f\left(\frac{[2^{N}\nu]+1}{2^{N}}\right)^{2^{N}\nu-[2^{N}\nu]},
\end{equation}
where by convention $f(x)=1$ when $x>1$. Observe that this is needed when $\nu=1$, because then $f\left(\frac{[2^{N}\nu]+1}{2^{N}}\right)=f\left(1+\frac{1}{2^{N}}\right).$ But this causes no problem because when $\nu=1$ the above inequality becomes an equality, with this convention.

Denoting the right hand side of (\ref{ineq_for_logconv_from_oam}) by $g_N(\nu),$ it is shown in the same paper that $g_{N+1}\leq g_{N}$ and that $g_N\to f$ as $N\to\infty$. This observation can be used to give further refinements of log-convex-type inequalities. For the next result, we use the notations $\alpha_N(\nu)=[2^{N}\nu]+1-2^{N}\nu, \beta_N(\nu)=\min\{\alpha_N(\nu),1-\alpha_N(\nu)\}, x_N(\nu)= \frac{[2^{N}\nu]}{2^{N}}, y_n(\nu)=\frac{[2^{N}\nu]+1}{2^{N}}$ and
$$K_N(f,\nu):=K\left(\sqrt[2^N]{\frac{f([2^{N}\nu]/2^N)}{f(([2^{N}\nu]+1)/2^N)}},2\right)^{\beta_N(\nu)}.$$
\begin{theorem}\label{further_ref_for_logconvex}
Let $f:[0,1]\to\mathbb{R}^+$ be log-convex and let $N\in\mathbb{N}$. Then
\begin{eqnarray*}
&&\hspace{-1cm}K_N(f,\nu)f(\nu)+S_N\left(\alpha_N(\nu);f(x_N(\nu)),f(y_N(\nu))\right)\\
&\leq& K_N(f,\nu)f\left(x_N(\nu)\right)^{\alpha_N(\nu)}f\left(y_N(\nu)\right)^{1-\alpha_N(\nu)}+
S_N\left(\alpha_N(\nu);f\left(x_N(\nu)\right),f\left(y_N(\nu)\right)\right)\\
&\leq&\alpha_N(\nu)f(x_N(\nu))+(1-\alpha_N(\nu))f(y_N(\nu))\\
&\leq&\nu f(1)+(1-\nu)f(0).
\end{eqnarray*}
\end{theorem}
\begin{proof}
The first inequality follows from (\ref{ineq_for_logconv_from_oam}), while the second follows from (\ref{the_first_refinement}). Thus, it remains to show the last inequality. Notice that
\begin{eqnarray*}
&&\hspace{-2cm}\alpha_N(\nu)f(x_N(\nu))+(1-\alpha_N(\nu))f(y_N(\nu))\\
&=&\alpha_N(\nu)f\left(x_N(\nu)\times 1+(1-x_N(\nu))\times 0\right)+\\
&&+(1-\alpha_N(\nu))f(y_N(\nu)\times 1+(1-y_N(\nu))\times 0)\\
&\leq&\alpha_N(\nu)\left\{x_N(\nu)f(1)+(1-x_N(\nu))f(0)\right\}+\\
&&+(1-\alpha_N(\nu))\left\{y_N(\nu)f(1)+(1-y_N(\nu))f(0)\right\}\\
&=&\nu f(1)+(1-\nu)f(0),
\end{eqnarray*}
where the last line is an immediate simplification of the given quantities.
\end{proof}

\section{Refined and Reversed Versions for Operators}
Now we state our main results about matrices.

To simplify the statement of the next result, we use the notation $\alpha_j(\nu)=\frac{k_j(\nu)}{2^{j-1}}.$ This result refines the corresponding inequalities in \cite{kittanehmanasreh} and \cite{zhao}.
\begin{theorem}
Let $A,B\in\mathbb{M}_n^{++}$, $N\in\mathbb{N}$ and $0\leq\nu\leq 1.$ Then
\begin{eqnarray*}
A\#_{\nu}B&+&\sum_{j=1}^{N}s_{j}(\nu)\left(A\#_{\alpha_j(\nu)}B+A\#_{2^{1-j}+\alpha_j(\nu)}B-2A\#_{2^{-j}+\alpha_j(\nu)}B\right)\leq A\nabla_{\nu}B.
\end{eqnarray*}
\end{theorem}
\begin{proof}
Let $b=1$ in (\ref{the_first_refinement}) and expand the summand to get
\begin{eqnarray}\label{needed_for_operator}
a^{\nu}+\sum_{j=1}^{N}s_j(\nu)\left(\sqrt[2^{j-1}]{a^{k_j(\nu)}}+\sqrt[2^{j-1}]{a^{1+k_j(\nu)}}-2\sqrt[2^{j}]{a^{2k_j(\nu)+1}}\right)\leq \nu a+(1-\nu).
\end{eqnarray}
Now let $X=A^{-\frac{1}{2}}BA^{-\frac{1}{2}}.$ Then $X\in \mathbb{M}_n^{++}$ and ${\text{sp}}(X)\subset (0,\infty)$. By functional calculus, we have
\begin{eqnarray*}
X^{\nu}+\sum_{j=1}^{N}s_j(\nu)\left(\sqrt[2^{j-1}]{X^{k_j(\nu)}}+\sqrt[2^{j-1}]{X^{1+k_j(\nu)}}-2\sqrt[2^{j}]{X^{2k_j(\nu)+1}}\right)\leq \nu X+(1-\nu)I.
\end{eqnarray*}
Conjugating both sides of this inequality by $A^{\frac{1}{2}}$ implies the result.
\end{proof}
For the next result, let $\beta_j(\nu)=2^{-j}k_j(2\nu)$ and $\gamma_j(\nu)=2^{1-j}k_j(2-2\nu).$ Applying Theorem \ref{our_first_reverse} and following similar steps as above give the following reversed version.
\begin{theorem}
Let $A,B\in\mathbb{M}_n^{++}$ and $N\in\mathbb{N}$. If $0\leq \nu\leq\frac{1}{2},$ then
\begin{eqnarray*}
A\nabla_{\nu}B&+&\sum_{j=1}^{N}s_j(2\nu)\left(A\#_{1-\beta_j(\nu)}B+A\#_{1+2^{-j}-\beta_j(\nu)}B-2A\#_{1-2^{-j-1}-\beta_j(\nu)}B\right)\\
&\leq&A\#_{\nu}B+2(1-\nu)\left(A\nabla B-A\#B\right).
\end{eqnarray*}
On the other hand, if $\frac{1}{2}\leq \nu\leq 1$, then
\begin{eqnarray*}
A\nabla_{\nu}B&+&\sum_{j=1}^{N}s_j(2-2\nu)\left(A\#_{\gamma_j(\nu)}B+A\#_{\gamma_j(\nu)+2^{1-j}}B-2A\#_{\gamma_j(\nu)+2^{-j}}B\right)\\
&\leq&A\#_{\nu}B+2\nu\left(A\nabla B-A\#B\right).
\end{eqnarray*}

\end{theorem}
We remark that these operator versions refine those in \cite{zhao}, where only the first two refining terms have been found there.\\

On the other hand, Theorem \ref{theorem_reverse_mojtaba} implies the following refined version of the corresponding result in \cite{mojtaba}.
\begin{theorem}
Let $A,B\in\mathbb{M}_n^{++}, \nu\geq 0$ and $N\in\mathbb{N}$. Then
\begin{eqnarray}
\nonumber A\nabla_{-\nu}B+\sum_{j=1}^{N}2^{j-1}\nu\left(A-2A\#_{2^{-j}}B+A\#_{2^{1-j}}B\right)\leq A\#_{-\nu}B.
\end{eqnarray}
\end{theorem}

Another reversed version can be obtained from Corollary \ref{reverse_after_mojtaba} as follows.
\begin{theorem}
Let $A,B\in\mathbb{M}_n^{++}, \nu\geq 0$ and $N\in\mathbb{N}$. Then
\begin{eqnarray*}
A\nabla_{-\nu}B&+&AB^{-1}A\sum_{j=1}^{N}\left(A\#_{1+\alpha_j(\nu)}B+A\#_{1+2^{1-j}+\alpha_j(\nu)}-2A\#_{1+2^{-j}+\alpha_j(\nu)}\right)\\
&\leq&A\#_{-\nu}B,
\end{eqnarray*}
where $\alpha_j(\nu)=2^{1-j}k_j(\nu).$
\end{theorem}

Applying Theorem \ref{refined_square} implies the following operator versions.
\begin{theorem}
Let $A,B\in\mathbb{M}_n^{++}$ and $N\geq 2$. If $0\leq\nu\leq \frac{1}{2}$, then
\begin{eqnarray*}
A\#_{2-2\nu}B&+&\sum_{j=2}^{N}s_{j}(\nu)\left(A\#_{\alpha_j(\nu)}B+A\#_{2^{1-j}+\alpha_j(\nu)}B-2A\#_{2^{-j}+\alpha_j(\nu)}B\right)\\
&\leq&\left( BA^{-1}B\right)\nabla_{2\nu}B.
\end{eqnarray*}
On the other hand, if $\frac{1}{2}\leq\nu\leq 1,$ then
\begin{eqnarray*}
A\#_{2-2\nu}B&+&\sum_{j=2}^{N}s_{j}(\nu)\left(A\#_{\alpha_j(\nu)}B+A\#_{2^{1-j}+\alpha_j(\nu)}B-2A\#_{2^{-j}+\alpha_j(\nu)}B\right)\\
&\leq& A\nabla_{2-2\nu}B.
\end{eqnarray*}
\end{theorem}
\begin{proof}
For $0\leq\nu\leq\frac{1}{2}$ and $a=1$, (\ref{squared_inequality_first}) can be written as
\begin{eqnarray*}
b^{2-2\nu}+\sum_{j=2}^{N}s_j(\nu)
\left(\sqrt[2^{j-1}]{b^{2^{j-1}-k_j(\nu)}}-\sqrt[2^{j-1}]{b^{2^{j-1}-k_j(\nu)-1}}\right)^2\leq (1-2\nu)b^2+2\nu b.
\end{eqnarray*}
Expanding the summand, letting $X=A^{-\frac{1}{2}}BA^{-\frac{1}{2}}$ and using functional calculus, we obtain the first inequality. Similar argument implies the second inequality.
\end{proof}
Operator refinements involving the Kantorovich constant can be obtained using the corresponding numerical versions. For example, (\ref{our_ref_kanto}) implies the following.
\begin{theorem}
Let $A,B\in\mathbb{M}_n^{++}, 0\leq \nu\leq 1$ and $N\in\mathbb{N}$. If $mI\leq A,B\leq MI$ for some $m,M>0$, then
\begin{eqnarray*}
K\left(h,2\right)^{\beta_N(\nu)}A\#_{\nu}B&+&\sum_{j=1}^{N}s_{j}(\nu)
\left(A\#_{\alpha_j(\nu)}B+A\#_{2^{1-j}+\alpha_j(\nu)}B-2A\#_{2^{-j}+\alpha_j(\nu)}B\right)\\
&\leq& A\nabla_{\nu}B,
\end{eqnarray*}
where $h=\sqrt[2^N]{\frac{M}{m}}$, $\beta_N(\nu)=\min\{1+[2^N\nu]-2^{N}\nu,2^{N}\nu-[2^{N}\nu]\}$ and $\alpha_j(\nu)=2^{1-j}k_j(\nu).$
\end{theorem}
\begin{proof}
Set $b=1$ in (\ref{our_ref_kanto}) to get
\begin{eqnarray*}
K(\sqrt[2^N]{a},2)^{\beta_N(\nu)}a^{\nu}+S_N(\nu;a,1)\leq \nu a+(1-\nu)1.
\end{eqnarray*}
If $a\in\left[\frac{m}{M},\frac{M}{m}\right]$, then using the facts that $K(t):=K(t,2)$ is an increasing function when $t>0$ and $K(t)=K(1/t),$ we obtain
$$K(h,2)^{\beta_N(\nu)}a^{\nu}+S_N(\nu;a,1)\leq \nu a+(1-\nu)1.$$ Now letting $X=A^{-\frac{1}{2}}BA^{-\frac{1}{2}}$, we obtain ${\text{sp}}(X)\subset \left[\frac{m}{M},\frac{M}{m}\right]$ because $mI\leq A,B\leq MI$. Then a standard functional calculus argument implies the required inequality.
\end{proof}
Observe that this theorem is a refinement of the corresponding results in \cite{zhao}.

\section{Some Hilbert-Schmidt Norm Inequalities}
In \cite{kittanehmanasreh}, it is proved that
\begin{eqnarray*}
\|A^{\nu}XB^{1-\nu}\|_2^2&+&s_1^2(\nu)\|AX-XB\|_2^2\leq\|\nu AX+(1-\nu)XB\|_2^2,
\end{eqnarray*}
when $A,B\in\mathbb{M}_n^{+}, X\in\mathbb{M}_n, 0\leq\nu\leq 1.$\\
The following theorem gives a refinement of this inequality \cite{kittanehmanasreh}.
\begin{theorem}
Let $A,B\in\mathbb{M}_n^{+}, X\in\mathbb{M}_n, 0\leq\nu\leq 1$ and $N\geq 2.$ Then
\begin{eqnarray*}
\|A^{\nu}XB^{1-\nu}\|_2^2&+&s_1^2(\nu)\|AX-XB\|_2^2\\
&+&\sum_{j=2}^{N}\left\|A^{\frac{k_j(\nu)}{2^{j-1}}}XB^{1-\frac{k_j(\nu)}{2^{j-1}}}  -A^{\frac{k_j(\nu)+1}{2^{j-1}}}XB^{1-\frac{k_j(\nu)+1}{2^{j-1}}}\right\|_2^2\\
&\leq&\|\nu AX+(1-\nu)XB\|_2^2.
\end{eqnarray*}
\end{theorem}
\begin{proof}
Since $A,B\in\mathbb{M}_n^{+}$, there are unitary matrices $U$ and $V$ such that
$$A=UD_1U^{*}\;{\text{and}}\;B=VD_2V^*$$ where $D_1={\text{diag}}(\lambda_i)$ and $D_2={\text{diag}}(\mu_i)$, $\lambda_i,\mu_i\geq 0.$ Letting $Y=U^*XV$, we have
\begin{eqnarray}
\nonumber \|A^{\nu}XB^{1-\nu}\|_2^2&=&\sum_{i,j}\left(\lambda_i^{\nu}\mu_j^{1-\nu}\right)^2|y_{ij}|^2,\\
\nonumber \|AX-XB\|_2^2&=&\sum_{i,j}\left(\lambda_i^{\nu}-\mu_j^{1-\nu}\right)^2|y_{ij}|^2,\\
\nonumber \left\|A^{\frac{k_j(\nu)}{2^{j-1}}}XB^{1-\frac{k_j(\nu)}{2^{j-1}}} \right.&-& \left.A^{\frac{k_j(\nu)+1}{2^{j-1}}}XB^{1-\frac{k_j(\nu)+1}{2^{j-1}}}\right\|_2^2\\
\nonumber&=&\sum_{i,j}\left(\lambda_i^{\frac{k_j(\nu)}{2^{j-1}}}\mu_j^{1-\frac{k_j(\nu)}{2^{j-1}}}  -\lambda_i^{\frac{k_j(\nu)+1}{2^{j-1}}}\mu_j^{1-\frac{k_j(\nu)+1}{2^{j-1}}}\right)^2|y_{ij}|^2,\\
\nonumber \|\nu AX+(1-\nu)XB\|_2^2&=&\sum_{i,j}\left(\nu \lambda_i+(1-\nu)\mu_j\right)^2|y_{ij}|^2.
\end{eqnarray}
These equations together with (\ref{squared_inequality_first}) implies the desired inequality.
\end{proof}
Following similar steps and applying Proposition \ref{reversed_square_reversed}, we obtain the following reversed version, refining the corresponding inequalities in \cite{kittanehmanasreh2}.
\begin{theorem}
Let $A,B\in\mathbb{M}_n^{+}, X\in\mathbb{M}_n,$ and $N\in\mathbb{N}.$ If $0\leq\nu\leq \frac{1}{2}$, then
\begin{eqnarray*}
\|\nu AX+(1-\nu)XB\|_2^2&+&\sum_{j=1}^{N}\left\|A^{1-\frac{k_j(2\nu)}{2^{j}}}XB^{\frac{k_j(2\nu)}{2^{j}}}  -A^{1-\frac{k_j(2\nu)+1}{2^{j}}}XB^{\frac{k_j(2\nu)+1}{2^{j}}}\right\|_2^2\\
&\leq&\|A^{\nu}XB^{1-\nu}\|_2^2+(1-\nu)^2\|AX-XB\|_2^2.
\end{eqnarray*}
On the other hand, if $\frac{1}{2}\leq\nu\leq 1,$ then
\begin{align*}
\|\nu AX+(1-\nu)XB\|_2^2&+\sum_{j=1}^{N}\left\|A^{\frac{k_j(2-2\nu)}{2^{j}}}XB^{1-\frac{k_j(2-2\nu)}{2^{j}}}  -A^{\frac{k_j(2-2\nu)+1}{2^{j}}}XB^{1-\frac{k_j(2-2\nu)+1}{2^{j}}}\right\|_2^2\\
&\leq\|A^{\nu}XB^{1-\nu}\|_2^2+\nu^2\|AX-XB\|_2^2.
\end{align*}
\end{theorem}
Further refinements of the Heinz inequality can be deduced for the Hilbert-Schmidt norm. First, we prove the following lemma, providing a significant refinement of the inequality
$$\left(a^{\nu}b^{1-\nu}+a^{1-\nu}b^{\nu}\right)^2+2\min\{\nu,1-\nu\}(a-b)^2\leq (a+b)^2,$$
proved in \cite{kittanehmanasreh}.
\begin{lemma}
Let $a,b>0$ and $N\in\mathbb{N}$. If $0\leq\nu\leq\frac{1}{2},$ then
\begin{eqnarray*}
\left(a^{\nu}b^{1-\nu}+a^{1-\nu}b^{\nu}\right)^2+2\nu(a-b)^2+S_N(2\nu;ab,a^2)+S_N(2\nu;ab,b^2)\leq (a+b)^2.
\end{eqnarray*}
On the other hand, if $\frac{1}{2}\leq\nu\leq 1,$ then
\begin{eqnarray*}
\left(a^{\nu}b^{1-\nu}+a^{1-\nu}b^{\nu}\right)^2&+&2(1-\nu)(a-b)^2+S_N(2-2\nu;ab,a^2)+S_N(2-2\nu;ab,b^2)\\
&\leq& (a+b)^2.
\end{eqnarray*}
\end{lemma}
\begin{proof}
If $0\leq\nu\leq\frac{1}{2},$ then by virtue of (\ref{the_first_refinement}), we have
\begin{eqnarray*}
&&\hspace{-1cm}(a+b)^2-2\nu(a-b)^2\\
&=&\left[2\nu ab+(1-2\nu)a^2\right]+\left[2\nu ab+(1-2\nu)b^2\right]+2ab\\
&\geq& 2ab+\left[(ab)^{2\nu}(a^2)^{1-2\nu}+S_N(2\nu;ab,a^2)\right]+\left[(ab)^{2\nu}(b^2)^{1-2\nu}+S_N(2\nu;ab,b^2)\right]\\
&=&\left(a^{\nu}b^{1-\nu}+a^{1-\nu}b^{\nu}\right)^2+S_N(2\nu;ab,a^2)+S_N(2\nu;ab,b^2).
\end{eqnarray*}
This completes the proof for $0\leq\nu\leq\frac{1}{2}.$ Similar computations imply the desired inequality for $\frac{1}{2}\leq\nu\leq 1.$
\end{proof}
Now these inequalities entail the following refinement of the Heinz means inequality, providing a refinement of the inequality
$$\|A^{\nu}XB^{1-\nu}+A^{1-\nu}XB^{\nu}\|_2^2+2\min\{\nu,1-\nu\}\|AX-XB\|_2^2\leq\|AX+XB\|_2^2,$$ proved in \cite{kittanehmanasreh}.
\begin{theorem}
Let $A,B\in\mathbb{M}_n^{+}, X\in\mathbb{M}_n$ and $N\in\mathbb{N}$. If $0\leq\nu\leq\frac{1}{2},$ then
\begin{eqnarray*}
&&\hspace{-2cm}\|A^{\nu}XB^{1-\nu}+A^{1-\nu}XB^{\nu}\|_2^2+2\nu\|AX-XB\|_2^2\\
&+&\sum_{j=1}^{N}s_j(2\nu)\left\|A^{1-\frac{k_j(2\nu)}{2^j}}XB^{\frac{k_j(2\nu)}{2^j}}-A^{1-\frac{k_j(2\nu)+1}{2^j}}XB^{\frac{k_j(2\nu)+1}{2^j}}\right\|_2^2\\
&+&\sum_{j=1}^{N}s_j(2\nu)\left\|A^{\frac{k_j(2\nu)}{2^j}}XB^{1-\frac{k_j(2\nu)}{2^j}}-A^{\frac{k_j(2\nu)+1}{2^j}}XB^{1-\frac{k_j(2\nu)+1}{2^j}}\right\|_2^2\\
&\leq&\|AX+XB\|_2^2.
\end{eqnarray*}
If $\frac{1}{2}\leq\nu\leq 1,$ then
\begin{eqnarray*}
&&\hspace{-1cm}\|A^{\nu}XB^{1-\nu}+A^{1-\nu}XB^{\nu}\|_2^2+2(1-\nu)\|AX-XB\|_2^2\\
&+&\sum_{j=1}^{N}s_j(2-2\nu)\left\|A^{1-\frac{k_j(2-2\nu)}{2^j}}XB^{\frac{k_j(2-2\nu)}{2^j}}-A^{1-\frac{k_j(2-2\nu)+1}{2^j}}XB^{\frac{k_j(2-2\nu)+1}{2^j}}\right\|_2^2\\
&+&\sum_{j=1}^{N}s_j(2-2\nu)\left\|A^{\frac{k_j(2-2\nu)}{2^j}}XB^{1-\frac{k_j(2-2\nu)}{2^j}}-A^{\frac{k_j(2-2\nu)+1}{2^j}}XB^{1-\frac{k_j(2-2\nu)+1}{2^j}}\right\|_2^2\\
&\leq&\|AX+XB\|_2^2.
\end{eqnarray*}
\end{theorem}
Refined reverses of the Heinz means can be found as well using Theorem \ref{our_first_reverse}.
\begin{lemma}
Let $a,b>0$ and $N\in\mathbb{N}$. If $0\leq \nu\leq\frac{1}{4},$ then
\begin{eqnarray*}
&&\hspace{-1cm}(a+b)^2+S_N(4\nu;\sqrt{a^3b},ab)+S_N(4\nu;\sqrt{ab^3},ab)\\
&\leq&\left(a^{\nu}b^{1-\nu}+a^{1-\nu}b^{\nu}\right)^2+2\nu(a-b)^2+(1-2\nu)\left[\left(\sqrt{ab}-a\right)^2+\left(\sqrt{ab}-b\right)^2\right].
\end{eqnarray*}
If $\frac{1}{4}\leq \nu\leq \frac{1}{2},$ then
\begin{eqnarray*}
&&\hspace{-1.5cm}(a+b)^2+S_N\left(2-4\nu;\sqrt{a^3b},a^2\right)+S_N\left(2-4\nu;\sqrt{ab^3},b^2\right)\\
&\leq&\left(a^{\nu}b^{1-\nu}+a^{1-\nu}b^{\nu}\right)^2+2\nu(a-b)^2+2\nu\left[\left(\sqrt{ab}-a\right)^2+\left(\sqrt{ab}-b\right)^2\right].
\end{eqnarray*}
If $\frac{1}{2}\leq \nu\leq \frac{3}{4},$ then
\begin{eqnarray*}
&&\hspace{-1cm}(a+b)^2+S_N\left(4\nu-2;\sqrt{a^3b},a^2\right)+S_N\left(4\nu-2;\sqrt{ab^3},b^2\right)\\
&\leq&\left(a^{\nu}b^{1-\nu}+a^{1-\nu}b^{\nu}\right)^2+2(1-\nu)(a-b)^2+(2-2\nu)\left[\left(\sqrt{ab}-a\right)^2+\left(\sqrt{ab}-b\right)^2\right].
\end{eqnarray*}
If $\frac{3}{4}\leq \nu\leq 1,$ then
\begin{eqnarray*}
&&\hspace{-1cm}(a+b)^2+S_N(4-4\nu;\sqrt{a^3b},ab)+S_N(4-4\nu;\sqrt{ab^3},ab)\\
&\leq&\left(a^{\nu}b^{1-\nu}+a^{1-\nu}b^{\nu}\right)^2+2(1-\nu)(a-b)^2+(2\nu-1)\left[\left(\sqrt{ab}-a\right)^2+\left(\sqrt{ab}-b\right)^2\right].
\end{eqnarray*}
\end{lemma}
These inequalities entail reversed versions of the Heinz inequality. Since the idea is the same, we present the inequality for the interval $0\leq\nu\leq\frac{1}{4}.$
\begin{theorem}
Let $A,B\in\mathbb{M}_n^{+}, X\in\mathbb{M}_n$ and $N\geq 2$. If $0\leq\nu\leq\frac{1}{4},$ then
\begin{eqnarray*}
&&\|AX+XB\|_2^2+\sum_{j=2}^{N}s_j(4\nu)\left\|A^{\frac{1}{2}+\frac{k_j(4\nu)}{2^{j+1}}}XB^{\frac{1}{2}
-\frac{k_j(4\nu)}{2^{j+1}}}-A^{\frac{1}{2}+\frac{k_j(4\nu)+1}{2^{j+1}}}XB^{\frac{1}{2}-\frac{k_j(4\nu)+1}{2^{j+1}}}\right\|_2^2\\
&&\hspace{2.6cm}+\sum_{j=2}^{N}s_j(4\nu)\left\|A^{\frac{1}{2}
-\frac{k_j(4\nu)}{2^{j+1}}}XB^{\frac{1}{2}+\frac{k_j(4\nu)}{2^{j+1}}}-A^{\frac{1}{2}-\frac{k_j(4\nu)+1}{2^{j+1}}}XB^{\frac{1}{2}+\frac{k_j(4\nu)+1}{2^{j+1}}}\right\|_2^2\\
&&\leq\|A^{\nu}XB^{1-\nu}+A^{1-\nu}XB^{\nu}\|_2^2+2\nu\|AX-XB\|_2^2\\
&&\hspace{2.5cm}+(1-2\nu)\left(\|A^{\frac{1}{2}}XB^{\frac{1}{2}}-AX\|_2^2+\|A^{\frac{1}{2}}XB^{\frac{1}{2}}-XB\|_2^2\right)
\end{eqnarray*}
\end{theorem}

\end{document}